\documentclass[a4paper,reqno]{amsart}
\usepackage[utf8]{inputenc}
\usepackage{amsmath,enumitem}
\usepackage{booktabs,url}
\usepackage{amssymb,yhmath,mathrsfs,wasysym}
\usepackage{amsthm}
\usepackage{graphicx}
\usepackage{subfigure}
\usepackage[subfigure]{ccaption}
\usepackage{multicol}
\usepackage{tikz,pgf}
\usetikzlibrary{arrows,automata,shapes,decorations.pathmorphing,decorations.markings,fadings,patterns}

\captionnamefont{\footnotesize\bf}
\captiontitlefont{\footnotesize}
\renewcommand{\thesubfigure}{\thefigure} \makeatletter \renewcommand{\p@subfigure}{} \renewcommand{\@thesubfigure}{\bf Figure \thesubfigure:\hskip\subfiglabelskip} \makeatother

\theoremstyle{plain}
\newtheorem{theorem}{Theorem}

\newtheorem{lemma}{Lemma}

\theoremstyle{definition}

\newtheorem{exercise}{Exercise}

\newcommand{\np}[1]{\textbf{\textit{#1}}}
\newcommand{\De}[0]{\mathbb{D}}
\newcommand{\NN}{\mathbb{N}}

\newcommand{\RR}{\mathbb{R}}
\newcommand{\CC}[0]{\mathbb{C}}
\newcommand{\ie}[0]{\textnormal{i}}

\newcommand{\de}[0]{\mathrel{\mathop:}=}

\newcommand{\holC}[1]{\mathcal{O}\left(#1\right)}
\newcommand{\dif}[1]{\textnormal{d}#1}

\newcommand{\pr}[2]{\pi^{#1}\left(#2\right)}
\newcommand{\prd}[2]{\pi_{#1}\left(#2\right)}

\DeclareMathOperator{\inter}{Int}

\title{Elementary approach to the Hartogs extension theorem}
\author{Aleksander Simoni\v{c}}
\address{Faculty of Mathematics and Physics \\ University of Ljubljana \\ Jadranska 19 \\ 1000 Ljubljana \\ Slovenia}
\email{aleksander.simonic@student.fmf.uni-lj.si}

\begin{document}

\begin{abstract}
In this paper we present a proof of Hartogs' extension theorem, following T.~Sobieszek's paper from 2003. Hartogs' theorem provides a large class of domains where holomorphic functions have analytic continuation to larger domains, and is ``a several complex variables theorem'' in nature because its conclusion is false in the complex plane. Sobieszek's proof is quite remarkable because he uses, stated in his paper without proofs, only higher-dimensional identity principle for holomorphic functions and Cauchy's integral formula for compact sets. We proved this two theorems here, making this exposition self-contained. The only background required is an undergraduate course in real and complex analysis and in point-set topology.
\end{abstract}

\maketitle
\thispagestyle{empty}

\section{Introduction}

In this paper we present an ``elementary'' proof of the following celebrated result, which many complex analysts consider the birth of multivariable complex analysis.

\begin{theorem}[Hartogs' Extension Theorem]
\label{thm:hartogs}
Let $\Omega$ be a bounded domain in $\CC^n$ where $n\geq2$ and $K\subset\Omega$ is a compact subset, such that $\Omega\setminus K$ is again a domain. Then every holomorphic function $f$ on $\Omega\setminus K$ has a unique holomorphic extension $\tilde{f}$ to $\Omega$, that is $\tilde{f}|_{\Omega\setminus K}=f$.
\end{theorem}

This theorem is often used to explain the difference between the theory of several complex variables and the one variable theory. Its standard proof needs preparation beyond the scope of an undergraduate course in complex analysis. Fortunately, Tomasz Sobieszek \cite{Sob} discovered a proof in which only a  higher-dimensional \emph{identity principle} (Theorem \ref{thm:id.princ}) and \emph{Cauchy's integral formula for compact sets} (Theorem \ref{thm:cauchy.compact}) are used. The latter theorem is closely related to \emph{Cauchy's integral formula for rectangles} (Theorem \ref{thm:cauchy.rect}), the only difference being that now the values come from a given compact subset of a domain in the complex plane. Unfortunately, Sobieszek's proof of Hartogs' theorem does not seem to be widely known among complex analysts. To the author's knowledge, the only source which contains such a sketch of the proof with reference to \cite{Sob} is Jarnicki and Pflug's textbook \cite[p.~58]{JP1}.

Throughout this paper we use the projection map $\pi^n\colon\CC^n\to\CC^{n-1}$ defined by $\pi^n(z_1,\ldots,z_n)\de(z_1,\ldots,z_{n-1})$. Immediate generalization is the map $\pi^i$ where we ``skip the $i$th component''. Of course, the notation for standard projection to the $i$th component is $\pi_i$. Sobieszek's theorem may be stated in the following form.

\begin{theorem}
\label{thm:gen.hart}
Let $\Omega$ be a domain in $\CC^n$ where $n\geq2$ and $A\subset\Omega$ is a subset, such that $\Omega\setminus A$ is again a domain. For $z'\in\pi^i(A)$ define the set
\begin{equation}
\label{eq:fiber}
A_i(z')\de\left\{z\in A\colon \pi^i(z)=z'\right\}.
\end{equation}
If $\pi^i\left(\Omega\right)\setminus\pi^i\left(A\right)$ is a nonempty open set and $A_i(z')$ is a compact set for every $z'\in\pi^i\left(A\right)$, then every holomorphic function $f$ on $\Omega\setminus A$ has a unique holomorphic extension $\tilde{f}$ to $\Omega$.
\end{theorem}

\noindent Using some general topology, it is easy to see that Theorem \ref{thm:hartogs} is really a special case of Theorem \ref{thm:gen.hart}. Recall that a domain in $\CC^n$ is by definition an open and path-connected\footnote{A subset $X\subset\CC^n$ (or more generally topological space $X$) is path-connected if for arbitrary points $x,y\in X$ a continuous map $\gamma\colon[0,1]\to X$ exists such that $\gamma(0)=x$ and $\gamma(1)=y$. From now on, we just say connected, as it is standard in analysis.} subset. Therefore $A$ must be a closed subset of $\Omega$. The projection $\pi^i$ is an open and continuous map, which is easy to verify by using one possible topological basis of $\CC^n$ - polydiscs (see Section \ref{sec:id.principle}). Compact sets in $\CC^n$ (or $\RR^{2n}$) are by \emph{Heine-Borel's theorem} (see, e.~g., \cite[Theorem 2.41]{RudinP}) a closed and bounded subsets. Conclusion is that if $A$ is a compact set, then $\pr{i}{A}$ and $A_i(z')$ are compact for every $z'\in\pr{i}{A}$, providing a nonempty open set $\pr{i}{\Omega}\setminus\pr{i}{A}$. The conditions of Theorem \ref{thm:gen.hart} are thus satisfied.

We choose to present the proof of Theorem \ref{thm:gen.hart} rather than just the proof of Hartogs' extension theorem. This choice is based on the fact that there is no essential difference between the two proofs because the main techniques are the same. Also, there are some other benefits, for instance the set $A$ needs not to be compact. As an example we demonstrate in the next section that Hartogs' figures, important domains in $\CC^n$ defined by \eqref{eq:hart.dom}, have the extension property. We must say that unboundedness of a domain $\Omega$ in Sobieszek's theorem is not a generalization of Hartogs' extension theorem. We stated ``bounded version'' of Theorem \ref{thm:hartogs} because then it is equivalent to Bochner's extension theorem (Theorem \ref{thm:bochner}). But classical proof of Theorem \ref{thm:hartogs} is also valid for unbounded domains, compare proofs of \cite[Theorem 1.2.6]{Kr1} and \cite[Theorem 5.4.4]{SchSCV}.

In this article, we provide a proof of Theorem \ref{thm:gen.hart} with as few sets of techniques as possible. To accomplish this task, we give a proof of the multidimensional identity principle (Section \ref{sec:id.principle}) and a proof of Cauchy's integral formula for compact sets (Section \ref{sec:cauchy}) assuming one-dimensional identity principle and Cauchy's integral formula for rectangles to be known. We start with some historical comments on Theorem \ref{thm:hartogs} and some of its consequences, where we inspect Hartogs' original approach, and finish with Section \ref{sec:proof} where we present the proof.

\section{Motivation}
\label{sec:motivation}

Assuming that the reader is familiar with the definition of a holomorphic function of one complex variable, we start with a higher dimensional definition. There are many definitions, but probably the simplest one is where we demand holomorphicity in each variable. Let $U$ be an open subset in $\CC^n$ where $n\geq2$. The function $f\colon U\to\CC$ is \np{separately holomorphic in $z_1$} if for every $z_2,\ldots,z_n\in\CC$ the one-variable function
\[
   \CC\ni\zeta\mapsto f\left(\zeta,z_2,\ldots,z_n\right)
\]
is holomorphic on the set $\{\zeta\in\CC\colon (\zeta,z_2,\ldots,z_n)\in U\}$. The function $f$ is \np{separately holomorphic} if it is separately holomorphic in every variable. Finally, we say that $f$ is \np{holomorphic} if it is continuous and separately holomorphic\footnote{The hypothesis about continuity is superfluous, which is another important Hartogs' contribution. The proof is not simple, see, e.~g., ~\cite[pp.~107-110]{Kr1}.}. The common notation for a family of holomorphic functions on $U$ is $\holC{U}$. Write $z_i=x_i+\ie y_i$ where $x_i$ and $y_i$ are real variables. Remember the relations from the one-variable theory
\begin{equation}
\label{eq:basic}
f_{z_i}\de\frac{\partial f}{\partial z_i} = \frac{1}{2}\left(\frac{\partial f}{\partial x_i} - \ie\frac{\partial f}{\partial y_i}\right)
\;\; \textrm{and} \;\;
f_{\bar{z}_i}\de\frac{\partial f}{\partial \bar{z}_i} = \frac{1}{2}\left(\frac{\partial f}{\partial x_i} + \ie\frac{\partial f}{\partial y_i}\right)
\end{equation}
for $f$ being a continuously differentiable function on $U$. Then $f$ is holomorphic if and only if $f_{\bar{z}_i}\equiv0$ on $U$ for every $1\leq i\leq n$.

Theorem \ref{thm:hartogs} is attributed to a German mathematician \textbf{Friedrich M.~Hartogs} (1874--1943) and is also known as Hartogs' Kugelsatz or just Kugelsatz. This name comes from two German words, namely  ``die Kugel'' (i.~e., a ball) and ``der Satz'' (i.~e., a sentence, a theorem). It is difficult to trace where and when the name first appeared. Although it might have circulated among German mathematicians before (i.~e., in the first half of the 20th century), it surely appeared in the second edition of Behnke and Thullen's book \cite{BehThu} on p.~63 in the form equivalent to Theorem \ref{thm:hartogs}.

\begin{theorem}[Bochner's Extension Theorem]
\label{thm:bochner}
Every holomorphic function, defined in a connected neighbourhood of the boundary $\partial\Omega$ of a bounded domain $\Omega\subset\CC^n$, where $n\geq2$, has a unique holomorphic extension on $\Omega$.
\end{theorem}

\noindent Why equivalent? Take a set $K\de\Omega\setminus U$ where $U\supset\partial\Omega$ is connected neighborhood of that function. Then $K$ is a compact in $\Omega$ and $\Omega\setminus K=\Omega\cap U$ is a domain. Therefore, Hartogs' theorem implies Bochner's theorem. On the other hand, given a compact set $K\subset\Omega$, there exists a neighbourhood $U\subset\Omega$ of $K$ such that $\partial U\subset\Omega\setminus K$. Therefore, Bochner's extension theorem for $U$ implies Hartogs' extension theorem for $\Omega$. For a proof of Theorem \ref{thm:bochner}, not relying on Theorem \ref{thm:hartogs}, see \cite{Range} or \cite{Range2}. The first reference is an expository paper on multidimensional complex analysis and the second reference discusses theorems connected to Bochner's theorem.

This interpretation of Hartogs' extension theorem is even more consistent with Hartogs' original approach. Given two bounded domains $\Omega_1,\Omega_2\subset\CC$ and a function $f$ defined in $\Omega_1\times\Omega_2$ and holomorphic on an open subset $U\subset\CC^2$ containing $\partial\left(\Omega_1\times\Omega_2\right)$, Hartogs proved in 1906 (see his paper \cite[p.~231]{Hart}) that $f$ is also holomorphic in $\Omega_1\times\Omega_2$. For $r>0$ define
\[
   \De^2(r)=\{(z_1,z_2)\in\CC^2\colon|z_1|<r,|z_2|<r\}.
\]
Today, Hartogs' idea is usually demonstrated in the case of a unit bidisc $\De^2\de\De^2(1)=\De\times\De$. Assume that holomorphic function $f$ is defined on the open set
\[
   U\de \De^2(1+\varepsilon)\setminus \overline{\De}^2(1-\varepsilon)
\]
for a number $\varepsilon>0$ (see Figure \ref{fig:hart1}). Because $\overline{\De}^2(1-\varepsilon)$ is a compact set in $\De^2(1+\varepsilon)$ and $U$ is connected, it follows from Theorem \ref{thm:hartogs} that every $f\in\holC{U}$ has a unique holomorphic extension $\tilde{f}$ to the domain $\De^2(1+\varepsilon)$.

\begin{figure}[h!]
\subfigure[][The open set $U$ indicated by three shades of grey contains a boundary $\partial\De^2$ of the unit bidisc $\De^2$. The domain $\De^2\cap U$ is indicated by darker shades of grey while $U'$ is the most dark grey.]
{\label{fig:hart1}
\begin{tikzpicture}[]
\clip (-1.1,-0.5) rectangle (4.7,3.8);
\fill[fill=gray,fill opacity=0.1] (0,0)--(2.9,0)--(2.9,2.9)--(0,2.9)--(0,0);
\fill[fill=gray,fill opacity=0.2] (0,0)--(2.3,0)--(2.3,2.3)--(0,2.3)--(0,0);
\fill[fill=gray,fill opacity=0.3] (1.7,0)--(2.3,0)--(2.3,2.3)--(1.7,2.3)--cycle;
\fill[color=white,fill=white] (0,0)--(1.7,0)--(1.7,1.7)--(0,1.7)--(0,0);
\draw[dash pattern=on 1pt off 1pt] (1.7,0)--(1.7,1.7)--(0,1.7);
\draw[dash pattern=on 1pt off 1pt] (1.7,1.7)--(1.7,2.3);
\draw[line width=1pt] (2.3,0)--(2.3,2.3)--(0,2.3);
\draw[dash pattern=on 1pt off 1pt] (2.9,0)--(2.9,2.9)--(0,2.9);
\draw (0,0)--(3.6,0);
\draw (0,0)--(0,3.6);
\draw (-0.1,1.7)--(0.1,1.7);
\draw (-0.1,2.3)--(0.1,2.3);
\draw (-0.1,2.9)--(0.1,2.9);
\draw (1.7,-0.1)--(1.7,0.1);
\draw (2.3,-0.1)--(2.3,0.1);
\draw (2.9,-0.1)--(2.9,0.1);
\begin{scriptsize}
\draw (-0.2,-0.2) node[] {$0$};
\draw (3.6,0.3) node[] {$|z_1|$};
\draw (0.3,3.6) node[] {$|z_2|$};
\draw (2.3,-0.3) node[] {$1$};
\draw (3.15,-0.3) node[] {$1+\varepsilon$};
\draw (1.45,-0.3) node[] {$1-\varepsilon$};
\draw (-0.3,2.3) node[] {$1$};
\draw (-0.5,2.9) node[] {$1+\varepsilon$};
\draw (-0.5,1.7) node[] {$1-\varepsilon$};
\draw (0.85,0.85) node[] {$\De^2(1-\varepsilon)$};
\draw (2,1.2) node[] {$U'$};
\draw (1.4,2.49) node[] {$\partial\De^2$};
\end{scriptsize}
\end{tikzpicture}
}
\hfill
\subfigure[][Complex two-dimensional Hartogs' figure $\mathcal{H}(q_1,q_2)$.]
{\setcounter{figure}{2} \label{fig:hart2}
\begin{tikzpicture}[]
\clip (-1.1,-0.5) rectangle (4.7,3.8);
\fill[color=blue,fill=gray,fill opacity=0.1] (0,0)--(2.9,0)--(2.9,2.9)--(1.7,2.9)--(1.7,1.9)--(0,1.9)--(0,0);
\draw[dash pattern=on 1pt off 1pt] (2.9,0)--(2.9,2.9);
\draw[dash pattern=on 1pt off 1pt] (2.9,2.9)--(1.7,2.9);
\draw[dash pattern=on 1pt off 1pt] (1.7,2.9)--(1.7,1.9);
\draw[dash pattern=on 1pt off 1pt] (1.7,1.9)--(0,1.9);
\draw (-0.1,2.9)--(0.1,2.9);
\draw (-0.1,1.9)--(0.1,1.9);
\draw (1.7,-0.1)--(1.7,0.1);
\draw (2.9,-0.1)--(2.9,0.1);
\draw (0,0)--(3.6,0);
\draw (0,0)--(0,3.6);
\begin{scriptsize}
\draw (-0.2,-0.2) node[] {$0$};
\draw (3.6,0.3) node[] {$|z_1|$};
\draw (0.3,3.6) node[] {$|z_2|$};
\draw (2.9,-0.3) node[] {$1$};
\draw (-0.3,2.9) node[] {$1$};
\draw (-0.3,1.9) node[] {$q_2$};
\draw (1.7,-0.3) node[] {$q_1$};
\draw (1.5,1) node[] {$\mathcal{H}(q_1,q_2)$};
\end{scriptsize}
\end{tikzpicture}
}
\end{figure}

\noindent The translated title of Hartogs' paper ``Some conclusions from Cauchy's integral formula for functions of several variables'' suggests that his main technique was Cauchy's integral representation formula. He observed that the function
\begin{equation}
\label{eq:hart}
F(z_1,z_2)\de \frac{1}{2\pi\ie} \oint_{\partial \De} \frac{f(z_1,\zeta)}{\zeta-z_2}\dif{\zeta}
\end{equation}
is holomorphic in whole $\De^2$ (see Figure \ref{fig:hart1}). Observe that for any $\zeta\in\partial\De$ the function $h$ on $\De^2$ defined by
\[
   h(z_1,z_2) \de \frac{f(z_1,\zeta)}{\zeta-z_2}
\]
is holomorphic, implying $h_{\bar{z}_1}=h_{\bar{z}_2}=0$. We could just move the derivation of $F$ to $\bar{z}_1$ and $\bar{z}_2$ under the integral, and simply obtain $F_{\bar{z}_1}=F_{\bar{z}_2}=0$. This process is correct since real analysis shows that the derivative of \emph{the integral with parameter} (for instance, the integral \eqref{eq:hart}) is an integral of the derivative if the derivative of the function is continuous \cite[Theorem 9.42]{RudinP}. This theorem is applicable after decomposing the function $h$ into real and imaginary parts and applying fundamental equations \eqref{eq:basic}. We thus obtain holomorphicity of $F$. Moreover, Cauchy's integral formula for a disc (just replace a rectangle with a disc in Theorem \ref{thm:cauchy.rect}) guarantees the equivalence of $F$ and $f$ on
\[
   U' \de \{z\in\CC\colon 1-\varepsilon<|z|<1\}\times\De.
\]
Therefore, the holomorphic function $F$ is the desired extension of $f$. This phenomenon of a simultaneous extension shocked mathematicians of those times who work on that problems. They realized that the theory of several complex variables is not such a straightforward generalization of the one variable theory since they could construct a holomorphic function without a holomorphic extension for any planar domain through any boundary point. This new discovery became one of the main research areas in this field, reaching its climax by characterizing the so-called domains of holomorphy.

\begin{exercise}
\label{ex:WML}
The well-known example of a holomorphic function on a disc where every boundary point is singular is given by the infinite series
\[
   \sum_{n=0}^{\infty}z^{2^n} = \sum_{n=0}^{N-1}z^{2^n} + \sum_{n=0}^{\infty}\left(z^{2^N}\right)^{2^n}.
\]
To prove this fact use the above decomposition and rotations $z\mapsto z\exp\left(2^{1-N}k\pi\ie\right)$ where $k$ are integers. The construction of such a function on arbitrary domain is carried out through \emph{Weierstrass' interpolation theorem}, which asserts the existence of a holomorphic function on a domain $\Omega$ with zeros on a given subset of $\Omega$ without accumulation points there, see \cite[Theorem 15.11]{Rud}. With help of a distance function and rational points in $\CC$ construct a subset of $\Omega$ which accumulation points \emph{are} $\partial\Omega$, or take a look at \cite[Corollary 3.3.3]{BerGay1991}.
\end{exercise}


Hartogs' domains, also called Hartogs' figures, are central to this theory. These domains are defined as
\begin{equation}
\label{eq:hart.dom}
\mathcal{H}(q_1,\ldots,q_n) = \left\{z\in\De^n \colon |z_1|>q_1 \; \textrm{or} \; |z_i|<q_i \; \textrm{for all} \; i\in\{2,\ldots,n\}\right\}
\end{equation}
where $\De^n$ is a natural generalization of $\De^2$ to $n$ dimensions, and $q_1,\ldots,q_n$ are numbers from the unit interval. In $\CC^2$ they have fine pictorial visualizations (see Figure \ref{fig:hart2}). It is easy to observe that $\mathcal{H}(q_1,\ldots,q_n) = \De^n \setminus A$ where
\[
   A \de \{z\in\CC\colon|z|\leq q_1\}\times\left\{\De^{n-1}\setminus\{z\in\CC^{n-1}\colon |z_2|<q_2,\ldots,|z_n|<q_n\}\right\}.
\]
It is tempting to use Theorem \ref{thm:hartogs} with $\Omega=\De^n$ and $K=A$. But the set $A$ is not compact. Applying a similar principle as before, we can simply obtain a holomorphic extension to $\De^2$ in the case of a complex two-dimensional Hartogs' figures.
The extension property for Hartogs' figures seems plausible. Fortunately, Theorem \ref{thm:gen.hart} comes to our rescue. The right choice for $i$ is $1$ because we have
\[
   \pr{1}{A} = \De^{n-1}\setminus\{z\in\CC^{n-1}\colon |z_2|<q_2,\ldots,|z_n|<q_n\}\neq\De^{n-1}=\pr{1}{\De^n},
\]
providing that $\pr{1}{\De^n}\setminus\pr{1}{A}$ is a nonempty open set and for every $z'\in\pr{1}{A}$ set $A_1(z')$ is a closed disc with a radius $q_1$. This set is clearly compact, thus satisfying the assumptions of Theorem \ref{thm:gen.hart}.

Let us highlight another consequence that Kugelsatz leads to. Namely, multivariable holomorphic functions do not have isolated zeros and singularities, which is completely false in one dimension. Indeed, if $f\in\holC{\Omega}$ has an isolated zero $a\in\Omega$, then $g\de1/f$ is holomorphic on $\Omega\setminus\{a\}$. But then, there is an extension $\tilde{g}\in\holC{\Omega}$ of the function $g$, which is impossible.

\begin{exercise}
In the same lines as before you can obtain that for $f\in\holC{\Omega}$, where $\Omega\subseteq\CC^n$ is a domain and $n\geq2$, the zero set $N(f)\de\left\{z\in\Omega\colon f(z)=0\right\}$ \emph{is not compact}. Excluding the trivial case $f\equiv0$, everything you need to verify is connectedness of the nonempty complement $\Omega\setminus N(f)$ in order to apply Kugelsatz. To do this, take arbitrary points $z$ and $w$ from that complement and define the \emph{complex line} as $\ell_{z,w}\de\left\{z+(w-z)\zeta\colon \zeta\in\CC\right\}$ through this two points. Observe that $\De^n(z,r)\cap\ell_{z,w}$ is some disc in the complex plane where $\De^n(z,r)$ is an arbitrary polydisc defined by \eqref{eq:polydisc}. What tells you the one-dimensional identity principle about the set $\De^n(z,r)\cap\ell_{z,w}\cap N(f)$? How is this useful for our purpose?
\end{exercise}

To conclude this section, let us mention that the statement of Hartogs' extension theorem in the case of a punctured bidisc $\De^2\setminus\{(0,0)\}$ was known to \textbf{Adolph Hurwitz} (1859-1919) nearly decade before. It appeared in his lecture at the first International Congress of Mathematicians in Z\"{u}rich in 1897 where he was also involved into organisation of that congress. Interestingly, $23$-old Hartogs, student of mathematics at University of Berlin those days, was a participant.

\section{The identity principle}
\label{sec:id.principle}

It is reasonable to expect that the identity principle also applies to holomorphic functions of several variables. Although this is true, it has to be assumed that the set of coincidences of two holomorphic functions is open, instead of having cluster points only, which is the case in the one variable theory. The basic counter-example is (globally a nonzero) a holomorphic function $f(z_1,z_2)=z_1(z_1-z_2)$ in $\CC^2$, which is identically zero on a closed set $\{0\}\times\CC\subset\CC^2$. This short section is entirely devoted to proving this basic theorem, relying only on the familiar one-dimensional identity principle.

Let $U_1\subseteq\CC^n$ and $U_2\subseteq\CC^m$ be two open subsets. We say that $f\colon U_1\to U_2$, $f(z)=(f_1(z),\ldots,f_m(z))$ where $z=(z_1,\ldots,z_n)$ is a holomorphic map if each $f_i(z)$ is a holomorphic function.

To generalize the notion of a unit polydisc $\De^n$, it is common to denote by $P^n$ a polydisc
\begin{equation}
\label{eq:polydisc}
\De^n(x,r)\de\left\{z\in\CC^n\colon \|z-x\|_{\infty}<r\right\}
\end{equation}
where $\|z\|_{\infty}=\max_{1\leq i\leq n}|z_i|$ for $z\in\CC^n$. Every polydisc $P^n$ is biholomorphic to the unit polydisc $\De^n=\De^n(0,1)$. Indeed,
\[
   (z_1,\ldots,z_n)\mapsto(rz_1+x_1,\ldots,rz_n+x_n)
\]
is a bijective holomorphic map with a  holomorphic inverse between $\De^n$ and $P^n$.

\begin{theorem}
\label{thm:id.princ}
Assume that two holomorphic functions on a domain $\Omega\subseteq\CC^n$ coincide on a nonempty open subset $U\subseteq\Omega$. Then, they are the same on the whole of $\Omega$.
\end{theorem}

\begin{proof}
Let $f_1,f_2\in\holC{\Omega}$ be two functions such that $f_1|_U\equiv f_2|_U$ for an open subset $U$ of a domain $\Omega\subseteq\CC^n$. Set $F\de f_1-f_2$. Then $F|_U\equiv0$. We must demonstrate that $F$ is identically zero on $\Omega$.

First, let us prove this for a bidisc $\De^2$. Choose an arbitrary point $(z_1,z_2)\in U$. By definition, $F$ is a holomorphic function on the set
\begin{equation}
\label{eq:slice}
\{\zeta\in\CC\colon(z_1,\zeta)\in\De^2\}=\De.
\end{equation}
By assumption, $F$ is identically zero on the open subset $\{\zeta\in\CC\colon(z_1,\zeta)\in U\}\subseteq\De$. The identity principle in one complex variable tells us $F|_{\{z_1\}\times\De}\equiv0$. Therefore, $F$ is identically zero on open subset $\prd{1}{U}\times\De\subseteq\De^2$. But from this it follows $F|_{\De\times\{\zeta\}}\equiv0$ for every $\zeta\in\De$, again by one-dimensional identity principle. Thus $F$ is zero on $\De^2$.

Proof for a polydisc $\De^n$ goes by induction on $n$; assume that we know that identity principle holds for $\De^{n-1}$ where $n\geq3$ and by the previous paragraph this was proved for $n=3$. Choose an arbitrary point $(z_1,\ldots,z_n)\in U\subseteq\De^n$. By similar arguments as before we obtain $F|_{\{(z_1,\ldots,z_{n-1})\}\times\De}\equiv0$ and $F$ is identically zero on open subset $\pr{n}{U}\times\De\subseteq\De^{n-1}\times\De$. By the induction argument we have $F|_{\De^{n-1}\times\{\zeta\}}\equiv0$ for every $\zeta\in\De$. Proof for polydiscs is complete.

In order to prove general case, choose an arbitrary point $z\in\Omega\setminus\overline{U}$. We want to show that $F(z)=0$. Because $\Omega$ is connected, there exists a path $\gamma$ with the initial point $w\in U$ and the end point $z$. This path can be covered by finitely many polydiscs $P_1,\ldots,P_k$, where $P_i\cap P_{i+1}\not=\emptyset$ for $i\in\{1,\ldots,k-1\}$, $P_1\subset U$ and $z\in P_k$. From the previous paragraph, $F|_{P_i}\equiv0$ for every $i\in\{1,\ldots,k\}$ and also $F(z)=0$.
\end{proof}

The standard proof of Theorem \ref{thm:id.princ} goes through multidimensional Taylor series, see \cite[Proposition 1.7.10]{JP1}. Different approach, presented in the textbook \cite[\S1.2.2]{SchSCV}, uses convexity of polydiscs. Observe that the reduction to polydiscs in our proof is crucial since the set \eqref{eq:slice} may not be connected in the general setting. This issue prevents us from using the one-dimensional identity principle directly.

\section{Cauchy's integral formula for compact sets}
\label{sec:cauchy}

In this section, we prove Cauchy's integral formula for compact sets (Theorem \ref{thm:cauchy.compact}) using Cauchy's integral formula for rectangles (Theorem \ref{thm:cauchy.rect}). Our exposition of the proof follows the ideas from Remmert's book \cite{Rem}. It should be noted that much of the material presented here is irrelevant when someone knows the general form of Cauchy's integral representation theorem in the language of index (or winding numbers); see, e.~g., ~\cite[Theorem 10.35]{Rud}. But as said, we want to avoid this concept here.

Let $R$ be a rectangle given by
\[
   R \de \left\{z\in\CC\colon a<\Re{z}<b, c<\Im{z}<d\right\}
\]
for real numbers $a<b$ and $c<d$. The boundary of $R$ consists of its four sides. Each side is a chord and thus easily parametrizated by the unit interval. For instance, the parametrization of the side $[a+c\ie,b+c\ie]$ could be the map
\[
   \gamma\colon [0,1] \ni t\mapsto (a+\ie c)(1-t)+(b+\ie c)t.
\]
Observe that when $t$ goes from $0$ to $1$, then $\gamma(t)$ goes from $a+c\ie$ to $b+c\ie$. The choice of parametrization the so-called \emph{orientation} is given on the side. The reader is invited to write down the parametrization of the boundary of $R$ with a positive (counter-clockwise) orientation.

For the proof and historic account on Cauchy's discovery (1831) of this milestone in complex function theory, consult \cite[\S 7.2]{Rem}. See also Chapter 6 of this book if someone is really uncomfortable with complex integration.

\begin{theorem}[Cauchy's integral formula for rectangles]
\label{thm:cauchy.rect}
Let $U\subseteq\CC$ be an open set and $\overline{R}\subset U$ be a closed, positively oriented rectangle. For every $f\in\holC{U}$ we have
\begin{equation}
\label{eq:cauchy.rect}
\frac{1}{2\pi\ie}\int_{\partial{R}} \frac{f(\zeta)}{\zeta-z}\dif{\zeta} = \left\{\begin{array}{ll}
                                                                                         f(z), & z\in R, \\
                                                                                         0, & z\in U\setminus\overline{R}.
                                                                                       \end{array}\right.
\end{equation}
\end{theorem}

Let us investigate what happens if we stick together two rectangles $\overline{R}_1$ and $\overline{R}_2$ along two sides of equal length and sum integrals over these two rectangles? The most common answer would probably be that we get the integral over ``one big'' rectangle $\overline{R}_{12}$, since the two sides ``have eliminated'' each other. This really works, which can be formally deduced by writing down integrals for all the sides of rectangles. The formula \eqref{eq:cauchy.rect} over $\overline{R}_{12}$ is also true for every point $z_0\in\overline{R}_1\cap\overline{R}_2\cap R_{12}$. Take a sequence $\{z_n\}_{n=1}^\infty\subset R_1$ with the limit point $z_0$. For an arbitrary small number $\varepsilon>0$ there exists $N\in\NN$, such that $|z_n-z_0|<\varepsilon$ for $n>N$. Also, positive numbers $M_1$ and $M_2$ exist, such that $|f(\zeta)|<M_1$ and $|(\zeta-z_n)(\zeta-z_0)|>M_2$ for every $\zeta\in\partial R_{12}$. For $n>N$ we have
\begin{flalign*}
\left|f(z_n)-\int_{\partial R_{12}} \frac{f(\zeta)}{\zeta-z_0}\dif{\zeta}\right| &=
\left|\int_{\partial R_{12}}\frac{f(\zeta)}{\zeta-z_n}\dif{\zeta} - \int_{\partial R_{12}} \frac{f(\zeta)}{\zeta-z_0}\dif{\zeta}\right| \\
&=\left|z_n-z_0\right|\cdot\left|\int_{\partial R_{12}} \frac{f(\zeta)}{(\zeta-z_n)(\zeta-z_0)}\dif{\zeta}\right|<\frac{OM_1}{M_2}\varepsilon
\end{flalign*}
where $O$ is the length of $\partial R_{12}$, the constant, of course. Because $f(z_n)$ tends to $f(z_0)$ the conclusion follows.

Given an arbitrary compact subset $K$ of $U$, the idea is to construct an open subset $P$ of $U$ which contains this compact set and formula \eqref{eq:cauchy.rect} is correct there. This set $P$ will be a union of a finite number of squares, forming ``polygonal sets'' (see Figure \ref{fig:cauchy}). Construction with squares enables us to use Theorem \ref{thm:cauchy.rect}.

\begin{theorem}
\label{thm:cauchy.compact}
Let $U\subseteq\CC$ be an open set and $K\subset U$ a compact subset. Then there exists an open set $P$ such that $K\subset P\subset \overline{P} \subset U$ and
\begin{equation}
\label{ena:cauchy}
\frac{1}{2\pi\ie}\int_{\partial P} \frac{f(\zeta)}{\zeta-z}\dif{\zeta} = \left\{\begin{array}{ll}
                                                                                         f(z), & z\in P, \\
                                                                                         0, & z\in U\setminus\overline{P}
                                                                                       \end{array}\right.
\end{equation}
for every $f\in\holC{U}$.
\end{theorem}

\begin{proof}
The analytic machinery for the proof depends on Theorem \ref{thm:cauchy.rect} and the ``continuity argument'' following it. The rest is point-set topology.

Because $K$ is a compact set in $U$, there exists $\varepsilon>0$ such that for every $z\in K$ the closed disc with the center in $z$ and the radius $\varepsilon$ is contained in $U$. Take a lattice, formed by parallels to real and imaginary axes, with the width and height $\varepsilon\sqrt{2}/2$. The lattice divides the plane into an infinite number of closed squares with diagonal $\varepsilon$. Because $K$ is compact, a finite set of such squares exists, say $\{Q_1,\ldots,Q_M\}$, such that $K$ is contained in the interior of set $\widetilde{P}$, defined by
\[
   \widetilde{P}\de\bigcup_{i=1}^M Q_i \subset U.
\]
The set $\widetilde{P}$ is closed in $U$ since it is a finite union of closed subsets of $U$. For each square choose those sides, if any, with the following property: for every point of the side its every neighbourhood in $U$ intersects $\widetilde{P}$ and $U\setminus\widetilde{P}$. Let $\mathscr{B}$ be a set containing those sides. Of course, $\mathscr{B}$ is not empty. By definition, union of elements of $\mathscr{B}$, denoted by $B$, is a boundary of $\widetilde{P}$ in $U$. Define $P\de\widetilde{P} \setminus B$. Then $P$ is nonempty and open, since we have a disjoint union
\[
   \widetilde{P}=\inter{\widetilde{P}}\sqcup B
\]
between the interior of $\widetilde{P}$ and $B$. By definition, $\partial P = B$.

Every square $Q_i$ equip with positive orientation. From the construction it follows
\[
   \sum_{i=1}^M \int_{\partial Q_i} \frac{f(\zeta)}{\zeta-z}\dif{\zeta} = \int_{\partial P} \frac{f(\zeta)}{\zeta-z}\dif{\zeta} \;\; \textrm{for every} \;\; z\in\bigcup_{i=1}^M \inter{Q_i}.
\]
But we already know that this is true for arbitrary $z\in P$.
\end{proof}

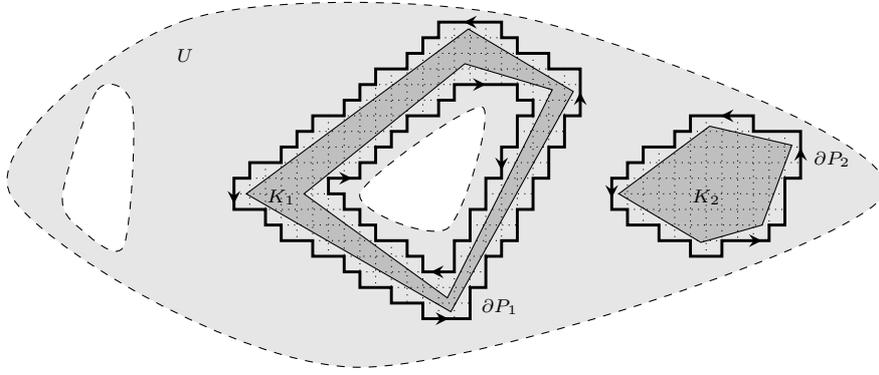
\begin{figure}[h!]
\centering
\begin{tikzpicture}[scale=2.3]

\tikzset{
    set arrow inside/.code={\pgfqkeys{/tikz/arrow inside}{#1}},
    set arrow inside={end/.initial=>, opt/.initial=},
    /pgf/decoration/Mark/.style={
        mark/.expanded=at position #1 with
        {
            \noexpand\arrow[\pgfkeysvalueof{/tikz/arrow inside/opt}]{\pgfkeysvalueof{/tikz/arrow inside/end}}
        }
    },
    arrow inside/.style 2 args={
        set arrow inside={#1},
        postaction={
            decorate,decoration={
                markings,Mark/.list={#2}
            }
        }
    },
}

\fill [gray!20] plot [smooth cycle] coordinates {(0,0) (1,1) (3,1) (5,0) (2,-1)};

\fill [gray!50]  (1.35,0) -- (2.62,0.95) -- (3.22,0.59) -- (2.52,-0.68) -- cycle;
\draw [black] (1.35,0) -- (2.62,0.95) -- (3.22,0.59) -- (2.52,-0.68) -- cycle;

\fill [gray!50]  (3.48,0) -- (4,0.39) -- (4.47,0.28) -- (4.3,-0.18) -- (3.95,-0.28) -- cycle;
\draw [black]  (3.48,0) -- (4,0.39) -- (4.47,0.28) -- (4.3,-0.18) -- (3.95,-0.28) -- cycle;

\fill [gray!20]  (1.68,0) -- (2.6,0.75) -- (3.1,0.6) -- (2.5,-0.6) -- cycle;
\draw [black] (1.68,0) -- (2.6,0.75) -- (3.1,0.6) -- (2.5,-0.6) -- cycle;

\fill [white] plot [smooth cycle] coordinates {(2,0) (2.7,0.5) (2.5,-0.2)};
\draw [black,dashed] plot [smooth cycle] coordinates {(0,0) (1,1) (3,1) (5,0) (2,-1)};
\draw [black,dashed] plot [smooth cycle] coordinates {(2,0) (2.7,0.5) (2.5,-0.2)};

\draw[dotted] (1.28,0) -- (1.82,0);
\draw[dotted] (1.37,-0.09) -- (1.91,-0.09);
\draw[dotted] (1.55,-0.18) -- (2,-0.18);
\draw[dotted] (1.73,-0.27) -- (2.18,-0.27);
\draw[dotted] (1.91,-0.36) -- (2.27,-0.36);
\draw[dotted] (2,-0.45) -- (2.36,-0.45);
\draw[dotted] (2.18,-0.54) -- (2.72,-0.54);
\draw[dotted] (2.36,-0.63) -- (2.63,-0.63);
\draw[dotted] (2.54,-0.45) -- (2.72,-0.45);
\draw[dotted] (2.54,-0.36) -- (2.81,-0.36);
\draw[dotted] (2.63,-0.27) -- (2.81,-0.27);
\draw[dotted] (2.63,-0.18) -- (2.9,-0.18);
\draw[dotted] (2.72,-0.09) -- (2.9,-0.09);
\draw[dotted] (2.72,0) -- (2.99,0);
\draw[dotted] (1.37,0.09) -- (1.82,0.09);
\draw[dotted] (2.81,0.09) -- (2.99,0.09);
\draw[dotted] (1.46,0.18) -- (2,0.18);
\draw[dotted] (2.81,0.18) -- (3.08,0.18);
\draw[dotted] (1.64,0.27) -- (2.09,0.27);
\draw[dotted] (2.9,0.27) -- (3.08,0.27);
\draw[dotted] (1.73,0.36) -- (2.18,0.36);
\draw[dotted] (2.9,0.36) -- (3.17,0.36);
\draw[dotted] (1.82,0.45) -- (2.27,0.45);
\draw[dotted] (2,0.54) -- (2.45,0.54);
\draw[dotted] (2.99,0.45) -- (3.26,0.45);
\draw[dotted] (2.99,0.54) -- (3.26,0.54);
\draw[dotted] (2,0.63) -- (2.54,0.63);
\draw[dotted] (2.9,0.63) -- (3.26,0.63);
\draw[dotted] (2.18,0.72) -- (3.08,0.72);
\draw[dotted] (2.36,0.81) -- (3.08,0.81);
\draw[dotted] (2.45,0.9) -- (2.81,0.9);

\draw[dotted] (1.37,0.09) -- (1.37,-0.09);
\draw[dotted] (1.46,0.18) -- (1.46,-0.18);
\draw[dotted] (1.55,0.27) -- (1.55,-0.18);
\draw[dotted] (1.64,0.27) -- (1.64,-0.27);
\draw[dotted] (1.73,0.36) -- (1.73,-0.27);
\draw[dotted] (1.82,0.45) -- (1.82,-0.36);
\draw[dotted] (1.91,0.54) -- (1.91,0.09);
\draw[dotted] (1.91,-0.09) -- (1.91,-0.36);
\draw[dotted] (2,0.54) -- (2,0.18);
\draw[dotted] (2,-0.18) -- (2,-0.45);
\draw[dotted] (2.09,0.72) -- (2.09,0.27);
\draw[dotted] (2.09,-0.18) -- (2.09,-0.54);
\draw[dotted] (2.18,0.72) -- (2.18,0.36);
\draw[dotted] (2.18,-0.27) -- (2.18,-0.54);
\draw[dotted] (2.27,0.81) -- (2.27,0.45);
\draw[dotted] (2.27,-0.36) -- (2.27,-0.63);
\draw[dotted] (2.36,0.81) -- (2.36,0.45);
\draw[dotted] (2.36,-0.45) -- (2.36,-0.63);
\draw[dotted] (2.45,0.9) -- (2.45,0.54);
\draw[dotted] (2.45,-0.45) -- (2.45,-0.72);
\draw[dotted] (2.54,0.99) -- (2.54,0.63);
\draw[dotted] (2.54,-0.45) -- (2.54,-0.72);
\draw[dotted] (2.63,0.99) -- (2.63,0.63);
\draw[dotted] (2.63,-0.27) -- (2.63,-0.63);
\draw[dotted] (2.72,0.99) -- (2.72,0.63);
\draw[dotted] (2.72,-0.09) -- (2.72,-0.45);
\draw[dotted] (2.81,0.9) -- (2.81,0.63);
\draw[dotted] (2.81,0.09) -- (2.81,-0.27);
\draw[dotted] (2.9,0.9) -- (2.9,0.63);
\draw[dotted] (2.9,0.27) -- (2.9,-0.09);
\draw[dotted] (2.99,0.81) -- (2.99,0.09);
\draw[dotted] (3.08,0.72) -- (3.08,0.27);
\draw[dotted] (3.17,0.72) -- (3.17,0.27);

\draw[dotted] (3.44,0) -- (4.43,0);
\draw[dotted] (3.53,0.09) -- (4.43,0.09);
\draw[dotted] (3.53,-0.09) -- (4.43,-0.09);
\draw[dotted] (3.71,-0.18) -- (4.34,-0.18);
\draw[dotted] (3.89,-0.27) -- (4.07,-0.27);
\draw[dotted] (3.62,0.18) -- (4.52,0.18);
\draw[dotted] (3.8,0.27) -- (4.52,0.27);
\draw[dotted] (3.89,0.36) -- (4.25,0.36);

\draw[dotted] (3.53,0.09) -- (3.53,-0.09);
\draw[dotted] (3.62,0.18) -- (3.62,-0.18);
\draw[dotted] (3.71,0.27) -- (3.71,-0.18);
\draw[dotted] (3.8,0.27) -- (3.8,-0.27);
\draw[dotted] (3.89,0.36) -- (3.89,-0.27);
\draw[dotted] (3.98,0.45) -- (3.98,-0.36);
\draw[dotted] (4.07,0.45) -- (4.07,-0.27);
\draw[dotted] (4.16,0.45) -- (4.16,-0.27);
\draw[dotted] (4.25,0.36) -- (4.25,-0.27);
\draw[dotted] (4.34,0.36) -- (4.34,-0.18);
\draw[dotted] (4.43,0.36) -- (4.43,0.09);

\draw[line width=1.1pt] (3.44,0.09) -- (3.44,-0.09) -- (3.53,-0.09) -- (3.53,-0.18) -- (3.71,-0.18) -- (3.71,-0.27) -- (3.89,-0.27) -- (3.89,-0.36) -- (4.07,-0.36) -- (4.07,-0.27) -- (4.34,-0.27) -- (4.34,-0.18) -- (4.43,-0.18) -- (4.43,0.09) -- (4.52,0.09) -- (4.52,0.36) -- (4.25,0.36) -- (4.25,0.45) -- (3.89,0.45) -- (3.89,0.36) -- (3.8,0.36) -- (3.8,0.27) -- (3.62,0.27) -- (3.62,0.18) -- (3.53,0.18) -- (3.53,0.09) -- cycle [arrow inside={end=stealth,opt={black,scale=1}}{0.03,0.37,0.57,0.74}];

\draw [line width=1.1pt] (1.28,0.09) -- (1.28,-0.09) -- (1.46,-0.09) -- (1.46,-0.18) -- (1.55,-0.18) -- (1.55,-0.27) -- (1.73,-0.27) -- (1.73,-0.36) -- (1.91,-0.36) -- (1.91,-0.45) -- (2,-0.45) -- (2,-0.54) -- (2.18,-0.54) -- (2.18,-0.63) -- (2.36,-0.63) -- (2.36,-0.72) -- (2.63,-0.72) -- (2.63,-0.54) -- (2.72,-0.54) -- (2.72,-0.36) -- (2.81,-0.36) -- (2.81,-0.27) -- (2.9,-0.27) -- (2.9,-0.09) -- (2.99,-0.09) -- (2.99,0.09) -- (3.08,0.09) -- (3.08,0.27) -- (3.17,0.27) -- (3.17,0.45) -- (3.26,0.45) -- (3.26,0.72) -- (3.08,0.72) -- (3.08,0.81) -- (2.9,0.81) -- (2.9,0.9) -- (2.81,0.9) -- (2.81,0.99) -- (2.45,0.99) -- (2.45,0.9) -- (2.36,0.9) -- (2.36,0.81) -- (2.27,0.81) -- (2.27,0.72) -- (2.09,0.72) -- (2.09,0.63) -- (2,0.63) -- (2,0.54) -- (1.91,0.54) -- (1.91,0.45) -- (1.73,0.45) -- (1.73,0.36) -- (1.64,0.36) -- (1.64,0.27) -- (1.55,0.27) -- (1.55,0.18) -- (1.37,0.18) -- (1.37,0.09) -- cycle [arrow inside={end=stealth,opt={black,scale=1}}{0.02,0.276,0.555,0.7}];

\draw [line width=1.1pt] (1.82,0.09) -- (2,0.09) -- (2,0.18) -- (2.09,0.18) -- (2.09,0.27) -- (2.18,0.27) -- (2.18,0.36) -- (2.27,0.36) -- (2.27,0.45) -- (2.45,0.45) -- (2.45,0.54) -- (2.54,0.54) -- (2.54,0.63) -- (2.9,0.63) -- (2.9,0.54) -- (2.99,0.54) -- (2.99,0.45) -- (2.9,0.45) -- (2.9,0.27) -- (2.81,0.27) -- (2.81,0.09) -- (2.72,0.09) -- (2.72,-0.09) -- (2.63,-0.09) -- (2.63,-0.27) -- (2.54,-0.27) -- (2.54,-0.45) -- (2.36,-0.45) -- (2.36,-0.36) -- (2.27,-0.36) -- (2.27,-0.27) -- (2.09,-0.27) -- (2.09,-0.18) -- (2,-0.18) -- (2,-0.09) -- (1.91,-0.09) -- (1.91,0) -- (1.82,0) -- cycle [arrow inside={end=stealth,opt={black,scale=1}}{0.03,0.32,0.53,0.75}];

\fill [white] plot [smooth cycle] coordinates {(0.3,0) (0.5,0.6) (0.7,0.5) (0.65,-0.3) (0.4,-0.2)};
\draw [black,dashed] plot [smooth cycle] coordinates {(0.3,0) (0.5,0.6) (0.7,0.5) (0.65,-0.3) (0.4,-0.2)};

\begin{scriptsize}
\draw (1.55,-0.02) node {$K_1$};
\draw (3.98,-0.02) node {$K_2$};
\draw (4.7,0.2) node {$\partial P_2$};
\draw (2.8,-0.65) node {$\partial P_1$};
\draw (1,0.8) node {$U$};
\end{scriptsize}
\end{tikzpicture}
\caption{Open set $U\subseteq\CC$ with compact set $K=K_1\cup K_2$ and open set $P=P_1\cup P_2$ containing $K$ and having boundary $\partial P_1\cup\partial P_2$. Little dotted squares illustrate how the set $P$ was constructed, see the proof of Theorem \ref{thm:cauchy.compact}.}
\label{fig:cauchy}
\end{figure}

The set $P$ from Theorem \ref{thm:cauchy.compact} is far from being unique. Assume that we have two such sets $P_1$ and $P_2$. Given a holomorphic $f$ on $U$, Theorem \ref{thm:cauchy.compact} guarantees the equality of integrals around $\partial P_1$ and $\partial P_2$ for values in $P_1\cap P_2$. But in the next section, we are confronted by the following problem: \emph{Assume that $f$ is defined and holomorphic only in the open set $U\setminus K$. Are integrals the same?} In this case, the integral \eqref{ena:cauchy} is also well-defined and represents a holomorphic function in $P$, but possibly different from $f$ because, on the contrary, with this method the extension of $f$ through a compact set would always exists. But this is in contradiction with the consequence of Weierstrass' theorem, see Exercise \ref{ex:WML}.

Equality of integrals is easy to obtain in the case of two rectangles as sets $P_1$ and $P_2$. The reader is invited to work out the proof. In the same lines we could produce a proof in the general case since sets are formed by rectangles. But rigorous proof, which is presented below, seems to be somehow technically difficult, but still quite elementary and natural.

In order to produce the proof, consider the compact subset
\[
   \widetilde{K}\de\overline{P_1\cup P_2\setminus P_1\cap P_2} \subset U\setminus K.
\]
By Theorem \ref{thm:cauchy.compact}, there exists an open set $P_3\subset\CC$ such that $\widetilde{K}\subset P_3\subset\overline{P}_3\subset U\setminus K$ and
\[
   f(z) = \frac{1}{2\pi\ie}\int_{\partial P_3} \frac{f(\zeta)}{\zeta-z}\dif{\zeta}
\]
for every $z\in P_3$. Now take $z$ from a nonempty open set $\left(P_1\cap P_2\right)\setminus\overline{P}_3$. We have
\begin{flalign*}
\int_{\partial P_1 - \partial P_2} \frac{f(\eta)}{\eta-z}\dif{\eta} &= \int_{\partial P_1 - \partial P_2} \int_{\partial P_3} \frac{f(\zeta)}{2\pi\ie(\zeta-\eta)(\eta-z)}\dif{\zeta}\dif{\eta} \\
&= - \int_{\partial P_3} \left(\frac{1}{2\pi\ie}\int_{\partial P_1 - \partial P_2} \left(\frac{1}{\eta-\zeta}-\frac{1}{\eta-z}\right)\dif{\eta}\right) \frac{f(\zeta)}{\zeta-z}\dif{\zeta}.
\end{flalign*}
Changing the order of integration is justified by the fact that the integrand is a continuous function and we integrate around a compact set\footnote{This is a weak (and most used) version of the celebrated Fubini's theorem. The reader can find it, stated in the sense of measure theory, in \cite[Chapter 8]{Rud}.}. The values $\zeta$ are from the boundary $\partial P_3$. It follows that $\zeta\in U\setminus\widetilde{K}$ and therefore $\zeta$ belongs to $U\setminus\overline{P_1\cup P_2}$ or $P_1\cap P_2$. In both cases, the inner integral in the above expression equals zero by Theorem \ref{thm:cauchy.compact}. We showed that
\begin{equation}
\label{eq:cauchy.equiv}
\frac{1}{2\pi\ie}\int_{\partial P_1} \frac{f(\zeta)}{\zeta-z}\dif{\zeta} = \frac{1}{2\pi\ie}\int_{\partial P_2} \frac{f(\zeta)}{\zeta-z}\dif{\zeta}
\end{equation}
for every $z\in\left(P_1\cap P_2\right)\setminus\overline{P}_3$. By the identity principle, the equality \eqref{eq:cauchy.equiv} holds for all $z\in P_1\cap P_2$.

\section{The proof}
\label{sec:proof}

In this section, we finally prove Theorem \ref{thm:gen.hart} and, consequently,  Hartogs' extension theorem. To begin with, let us review the assumptions of the theorem.

We have a domain $\Omega$ in $\CC^n$ where $n\geq2$, and a subset $A\subset\Omega$ such that $\Omega\setminus A$ is a domain. Furthermore, there exists a number $i\in\{1,\ldots,n\}$ such that:
\begin{enumerate}
\item $\pr{i}{\Omega}\setminus\pr{i}{A}$ is a nonempty open set\footnote{The openness is not fully guaranteed although $A$ is a closed subset of $\Omega$. This is because $\pi^i$ is not a closed map. A possible counter-example is provided by the set $A=\{(z_1,z_2)\in\CC^2\colon z_1z_2=1\}$ in $\Omega=\CC^2$.},
\item $A_i(z')$ is a compact set for every $z'\in\pr{i}{A}$,
\end{enumerate}
where $A_i(z')$ is the fiber \eqref{eq:fiber} and $\pi^i$ is the projection where we omit the $i$th component. From now on, this number $i$ is fixed.

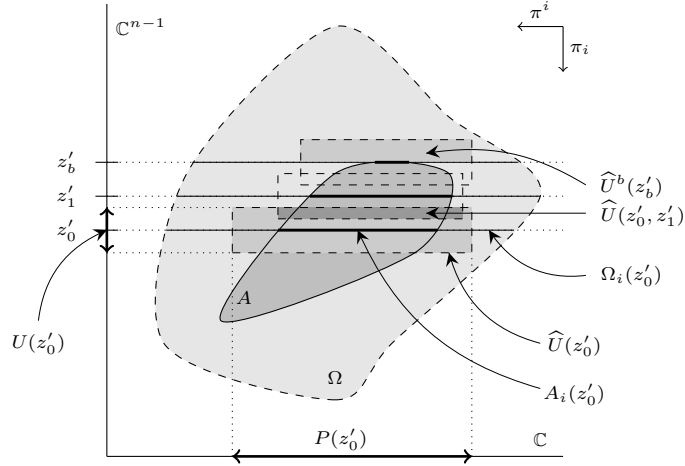
\begin{figure}[h!]
\centering
\begin{tikzpicture}[scale=1.5]
\draw (0,0) -- (4,0);
\draw (0,0) -- (0,4);
\fill [gray!20] plot [smooth cycle] coordinates {(0.5,1) (2,0.5) (2.5,1) (3.8,2.3) (3,3) (2,3.8) (0.7,2.5)};
\draw [dashed] plot [smooth cycle] coordinates {(0.5,1) (2,0.5) (2.5,1) (3.8,2.3) (3,3) (2,3.8) (0.7,2.5)};
\fill [gray!50] plot [smooth cycle] coordinates {(1,1.2) (2.7,1.8) (3,2.5) (2,2.5)};
\fill [gray!40] (1.7,2.6) -- (3.2,2.6) -- (3.2,2.8) -- (1.7,2.8) -- (1.7,2.6);
\fill [gray!40] (1.1,1.8) -- (3.2,1.8) -- (3.2,2.2) -- (1.1,2.2) -- (1.1,1.8);
\fill[gray!80] (1.5,2.1) -- (3.12,2.1) -- (3.12,2.2) -- (1.5,2.2) -- (1.5,2.1);
\draw plot [smooth cycle] coordinates {(1,1.2) (2.7,1.8) (3,2.5) (2,2.5)};

\draw (0.52,2) -- (3.57,2);
\draw[dotted] (0,2) -- (4,2);
\draw[line width=1.1pt] (1.5,2) -- (2.9,2);
\draw[line width=0.2pt,dashed] (1.1,1.8) -- (3.2,1.8) -- (3.2,2.2) -- (1.1,2.2) -- (1.1,1.8);

\draw (0.61,2.3) -- (3.8,2.3);
\draw[dotted] (0,2.3) -- (4,2.3);
\draw[line width=1.1pt] (1.78,2.3) -- (3.03,2.3);
\draw[line width=0.2pt,dashed] (1.5,2.1) -- (3.12,2.1) -- (3.12,2.5) -- (1.5,2.5) -- (1.5,2.1);

\draw (0.75,2.6) -- (3.6,2.6);
\draw[dotted] (0,2.6) -- (4,2.6);
\fill [semitransparent,gray!40] (1.7,2.6) -- (3.2,2.6) -- (3.2,2.8) -- (1.7,2.8) -- (1.7,2.6);
\draw[line width=1.1pt] (2.35,2.6) -- (2.65,2.6);
\draw[line width=0.2pt,dashed] (1.7,2.4) -- (3.2,2.4) -- (3.2,2.8) -- (1.7,2.8) -- (1.7,2.4);

\draw[->] (4,3.8) -- (3.6,3.8);
\draw[->] (4,3.8) -- (4,3.4);

\draw[decoration={markings,mark=at position 1 with
    {\arrow[scale=2,>=stealth]{>}}},postaction={decorate},line width=0.1pt] (3.8,0.6) to[bend left=25] (2.2,2);
\draw[decoration={markings,mark=at position 1 with
    {\arrow[scale=2,>=stealth]{>}}},postaction={decorate},line width=0.1pt] (3.8,1) to[bend left=30] (3,1.8);
\draw[decoration={markings,mark=at position 1 with
    {\arrow[scale=2,>=stealth]{>}}},postaction={decorate},line width=0.1pt] (-0.6,1.2) to[bend left=30] (0,2);
\draw[decoration={markings,mark=at position 1 with
    {\arrow[scale=2,>=stealth]{>}}},postaction={decorate},line width=0.1pt] (4.2,2.4) to[bend right=20] (2.8,2.68);
\draw[decoration={markings,mark=at position 1 with
    {\arrow[scale=2,>=stealth]{>}}},postaction={decorate},line width=0.1pt] (4.2,1.6) to[bend left=32] (3.35,2);
\draw[decoration={markings,mark=at position 1 with
    {\arrow[scale=2,>=stealth]{>}}},postaction={decorate},line width=0.1pt] (4.2,2.15) to (2.8,2.15);

\draw[<->,line width=1pt] (0,1.8) -- (0,2.2);
\draw[<->,line width=1pt] (1.1,0) -- (3.2,0);
\draw[dotted] (1.1,0) -- (1.1,1.8);
\draw[dotted] (3.2,0) -- (3.2,1.8);
\draw[dotted] (0,1.8) -- (1.1,1.8);
\draw[dotted] (0,2.2) -- (1.1,2.2);

\draw (-0.1,2) -- (0.1,2);
\draw (-0.1,2.3) -- (0.1,2.3);
\draw (-0.1,2.6) -- (0.1,2.6);

\begin{scriptsize}
\draw (3.8,0.15) node[] {$\CC$};
\draw (0.3,3.8) node[] {$\CC^{n-1}$};
\draw (3.8,3.95) node[] {$\pi^i$};
\draw (4.15,3.6) node[] {$\pi_i$};
\draw (1.2,1.4) node[] {$A$};
\draw (2,0.7) node[] {$\Omega$};
\draw (4.1,0.55) node[] {$A_i(z_0')$};
\draw (4.6,1.6) node[] {$\Omega_i(z_0')$};
\draw (4.7,2.15) node[] {$\widehat{U}(z_0',z_1')$};
\draw (-0.35,2) node[] {$z_0'$};
\draw (-0.35,2.3) node[] {$z_1'$};
\draw (-0.35,2.6) node[] {$z_b'$};
\draw (2.05,0.15) node[] {$P(z_0')$};
\draw (-0.6,1) node[] {$U(z_0')$};
\draw (4.1,1) node[] {$\widehat{U}(z_0')$};
\draw (4.6,2.4) node[] {$\widehat{U}^{b}(z_b')$};
\end{scriptsize}
\end{tikzpicture}
\caption{Schematic components in the proof of Hartogs' extension theorem.}
\label{fig:proof}
\end{figure}

Taking the following considerations into account, the reader is invited to examine Figure \ref{fig:proof}. Choose an arbitrary point $z_0'\in\pr{i}{A}$. By the second assumption above, the set $A_i(z_0')$ is compact and therefore $\prd{i}{A_i(z_0')}$ is also compact subset of $\CC$. Theorem \ref{thm:cauchy.compact} guarantees the existence of an open set $P(z_0')\subset\CC$, such that
\[
   \prd{i}{A_i(z_0')}\subset P(z_0')\subset\overline{P}(z_0')\subset\prd{i}{\Omega_i(z_0')}.
\]
If this open set is also good for small perturbations of $z_0'$ in $\CC^{n-1}$ such that the above equation is correct, then we can define an open set
\begin{equation}
\label{eq:open.set}
\widehat{U}(z_0') \de \left\{z\in\CC^n \colon \pr{i}{z}\in U(z_0'),\prd{i}{z}\in P(z_0')\right\}\subset\Omega
\end{equation}
with a neighbourhood $U(z_0')$ of $z_0'$ in $\CC^{n-1}$ where this perturbations are allowed. The precise statement and its existence is assured by the next lemma, which has the same purpose as Sobieszek's lemma \cite[Lemma 2]{Sob}.

\begin{lemma}
\label{lem:ind}
For every $z_0'\in\pr{i}{A}$ there exist an open set $P(z_0')\subset\CC$ and a neighborhood $U(z_0')\subset\CC^{n-1}$ of $z_0'$ such that
\begin{equation}
\label{eq:lemma}
\prd{i}{A_i(z')}\subset P(z_0') \subset \overline{P(z_0')} \subset\prd{i}{\Omega_i(z')}
\end{equation}
for every $z'\in U(z_0')$.
\end{lemma}

\begin{proof}
We are proving by contradiction. Assume that such $U(z_0')$ does not exist. Then for every neighborhood $V\subset\CC^{n-1}$ of $z_0'$ there exists $z'\in V$, such that $\prd{i}{A_i(z')}\not\subset P(z_0')$ or $\overline{P}(z_0')\not\subset\prd{i}{\Omega_i(z')}$. This means that there is a sequence $\{z_j'\}_{j=1}^\infty\subset\CC^{n-1}$ with the limit point $z_0'$ and the following property: for every $z_j'$ there exist $z_{1j}\in\prd{i}{A_i(z_j')}$ or $z_{2j}\in\overline{P}(z_0')$, such that $z_{1j}\notin P(z_0')$ or $z_{2j}\notin\prd{i}{\Omega_i(z_j')}$. Since the sets $\prd{i}{A_i(z_j')}$ and $\overline{P}(z_0')$ are compact for every $j\in\NN$, the sequences $\{z_{1j}\}$ and $\{z_{2j}\}$ have accumulation points in the latter sets. Therefore, we can assume that both sequences converge to $z_1\in\prd{i}{A_i(z_0')}\subset P(z_0')$ and $z_2\in\overline{P}(z_0')\subset\prd{i}{\Omega_i(z_0')}$, respectively. Since the set $P(z_0')$ is open and $\{z_{1j}\}\not\subset P(z_0')$, the limit point $z_1$ does not belong to $P(z_0')$. For the same reason, we must have $z_2\notin\prd{i}{\Omega_i(z_0')}$. We reached a contradiction since it has been assumed that such $U(z_0')$ is nonexistent.
\end{proof}

Thanks to Lemma \ref{lem:ind}, we know now for certain that such $\widehat{U}(z_0')$ exists. Before continuing with proof, let introduce an open set
\[
   \widehat{U}(z_1',z_2') \de \widehat{U}(z_1') \cap \widehat{U}(z_2')
\]
for arbitrary points $z_1',z_2'\in\pr{i}{A}$. This set is nonempty if and only if $U(z_1')$ and $U(z_2')$ have nonempty intersection and we can always choose such two points that this happens.

Let $f$ be a given holomorphic function on $\Omega\setminus A$. Again by Lemma \ref{lem:ind}, we can define the function $f_{z_0'}\colon \widehat{U}(z_0')\to\CC$ by
\begin{equation}
\label{eq:ext}
f_{z_0'}(z)\de \frac{1}{2\pi\ie}\int_{\partial P(z_0')} \frac{f\left(\pr{i}{z},\zeta\right)}{\zeta-\prd{i}{z}}\dif{\zeta}.
\end{equation}
The same argument as in Section \ref{sec:motivation} implies holomorphicity of $f_{z_0'}$ on $\widehat{U}(z_0')$. You can guess now that the thread of the proof is using functions \eqref{eq:ext} to achieve the following three assertions:
\begin{enumerate}
\item construction of a neighborhood $U$ of $A$ in $\Omega$ and,
\item construction of a holomorphic function on $U$ such that,
\item it coincides with $f$ on $U\setminus A$.
\end{enumerate}
It is evident that such a function will be an extension of $f$ to whole domain $\Omega$ and by the identity principle (Theorem \ref{thm:id.princ}) this extension is unique. Thus proving above assertions will provide the proof of Theorem \ref{thm:gen.hart}.

Neighborhood $U$ of $A$ consists of sets \eqref{eq:open.set}. Precisely, $U$ is a union of sets $\widehat{U}(z')$ for all $z'\in\pr{i}{A}$.

The second assertion will be complete if functions \eqref{eq:ext} have ``gluing property''. This means that $f_{z_1'}$ and $f_{z_2'}$ coincide on $\widehat{U}(z_1',z_2')$ for arbitrary points $z_1'$ and $z_2'$ from $\pr{i}{A}$. To see this, take arbitrary $z\in\widehat{U}(z_1',z_2')$ and define $z'\de\pr{i}{z}$. Then
\[
   \prd{i}{A_i(z')}\subset P(z_1') \cap P(z_2').
\]
We have a compact set, contained in two ``polygonal sets'', situation already discussed at the end of Section \ref{sec:cauchy}. By \eqref{eq:cauchy.equiv}, this means that $f_{z_1'}(z)=f_{z_2'}(z)$.

The third assertion is the most difficult one to show but it is the most important one, too. The reader is again invited to examine Figure \ref{fig:proof}. Choose arbitrary point $z_b'$ from the boundary of $\pr{i}{A}$ in $\pr{i}{\Omega}$ and define the open subset $\widehat{U}^b(z_b')$ of $\widehat{U}(z_b')$ by
\[
   \widehat{U}^b(z_b') \de \left\{z\in\CC^n \colon \pr{i}{z}\in U(z_b')\setminus\pr{i}{A},\prd{i}{z}\in P(z_b')\right\}.
\]
By the first assumption of Theorem \ref{thm:gen.hart}, this set is nonempty. Crucial point here is that this set \emph{is contained} in the complement of $A$ in $\Omega$. Remember that here we have function $f$, fully available to use Theorem \ref{thm:cauchy.compact} on it to obtain coincidence of $f$ and $f_{z_b'}$ on $\widehat{U}^b(z_b')$. By the identity principle, functions $f$ and $f_{z_b'}$ also coincide on $\widehat{U}(z_b')\setminus A$. The proof of the third assertion will be complete if we demonstrate this property for all points in $\pr{i}{A}$ rather than just for boundary points.

Take arbitrary $z'\in\pr{i}{A}$. Then there exists $z_b'$ from the boundary of $\pr{i}{A}$ in $\pr{i}{\Omega}$, and path $\gamma\colon[0,1]\to\pr{i}{A}$ with the initial point $z'$ and the end point $z_b'$. We now proceed as in the proof of the identity theorem, see Section \ref{sec:id.principle}. Choose points $z_1',\ldots,z_{k-1}'\in\gamma{\left([0,1]\right)}$ such that the set
\[
   \left\{U(z'=z_0'),U(z_1'),\ldots,U(z_{k-1}'),U(z_b'=z_k')\right\}
\]
covers this path. Then $\widehat{U}(z_j',z_{j+1}')\setminus A$ is for every $j\in\{0,\ldots,k-1\}$ a nonempty open subset of $\Omega\setminus A$. We know from the previous two paragraphs that
\[
   f \equiv f_{z_k'} \;\; \textrm{and} \;\; f_{z_k'} \equiv f_{z_{k-1}'} \;\; \textrm{on} \;\; \widehat{U}(z_{k-1}',z_{k}')\setminus A.
\]
By the identity principle, functions $f$ and $f_{z_{k-1}'}$ coincide on $\widehat{U}(z_{k-1}')\setminus A$. By continuing this process, we finally get to the point $z_0'=z'$ and the proof is complete.

\providecommand{\bysame}{\leavevmode\hbox to3em{\hrulefill}\thinspace}
\providecommand{\MR}{\relax\ifhmode\unskip\space\fi MR }
\providecommand{\MRhref}[2]{%
  \href{http://www.ams.org/mathscinet-getitem?mr=#1}{#2}
}
\providecommand{\href}[2]{#2}


\end{document}